\documentclass{ifacconf}

\usepackage{color}
\usepackage{mathrsfs,amssymb,natbib}
\usepackage{graphicx}      
\usepackage{natbib}
\usepackage{picins}
\usepackage{booktabs}
\usepackage{color}
\usepackage{graphicx}  
\usepackage{subcaption}
\usepackage[cmex10]{amsmath}
\usepackage{mathrsfs,amssymb,natbib}
\usepackage{float}
\usepackage{adjustbox}
\usepackage{array}
\usepackage{cases}
\usepackage{wrapfig}
\usepackage{graphicx,picinpar}
\usepackage{graphicx}
\usepackage{nccbbb}
\usepackage{enumerate}
\newtheorem{lemma}{Lemma}[section]
\usepackage{titlesec}
\titlespacing*{\section}{0pt}{*0.4}{*0.4}
\titlespacing*{\subsection}{0pt}{*0.3}{*0.3}

\usepackage{epstopdf}

\begin{document}
\begin{frontmatter}

\title{
 Boundary Stabilization for  the Rayleigh Beam System  under Event-triggered Controls
 }


\author[First]{Siwen Wang}
\author[First]{Yi Cheng}
\author[Second]{Yuhu Wu}
\author[Three]{Yongxin Wu}
\author[Four]{Wen Kang}
\address[First]{Department of Mathematics, Bohai University,
   Jinzhou, 121013, China (e-mail: wangsiwen011107@163.com;chengyi407@bhu.edu.cn)}
\address[Second]{School of Control Science and Engineering, Dalian University of Technology,
   Dalian, 116024, China (e-mail: wuyuhu@dlut.edu.cn)}
\address[Three]{Universit\'e Marie et Louis Pasteur, SUPMICROTECH, CNRS, Institut FEMTO-ST, Besan\c{c}on 25000, France (e-mail: yongxin.wu@femto-st.fr)}
\address[Four]{School of Mathematics and Statistics, Beijing Institute of Technology,
	Beijing, 100083, China (e-mail: kangwen@bit.edu.cn)}
	\thanks{This work was partially supported by the National Natural Science Foundation of China under grant nos.U24A20263, 62173062, 12131008, and U23B2033; National Key R\&D Program of China under grant no.2024YFA1013101, and the Natural Science Foundation of Liaoning Province of China under grant no.2025-MS-278. }
\begin{abstract}                 
In this paper, we propose two event-triggered control laws incorporating an event-triggering mechanism to tackle the boundary stabilization for the Rayleigh beam system.
Under this event-triggered controls, a sufficient condition for parameter determination is constructed to guarantee the exponential stability of the closed-loop system by using the integral multiplier technique and energy perturbation method, wherein the desired exponential decay rate can be precisely determined. Numerical examples are presented to demonstrate the efficacy of the event-triggered control methodology.
\end{abstract}

\end{frontmatter}

\section{Introduction}
Event-triggered control has attracted much attention owing to the advantage of saving communication resources. The key advantage of event-triggered control resides in its ``on-demand action'' mechanism, which called event-triggering mechanism. The idea of this mechanism is to decide the timing of control updates by comparing state errors with system energy, which realizing efficient resource utilization. Applications of event-triggered control are concentrated on hyperbolic equations \citep{Espitia2016Event, Diagne2021Boundary, Koudohode2022Event, Wang2022hyperbolic, Baudouin2019, Wang2025traffic,Zhang2025,Rathnayake2025Global}, parabolic equations \citep{Selivanov2016Distributed, Espitia2021Event, Rathnayake2025Performance}, parabolic partial differential equation-ordinary differential equation (PDE-ODE) cascades \citep{Wang2025traffic} and hyperbolic PDE-ODE cascades \citep{Wang2022adaptive}. Additionally, event-triggered
control strategies for other types of PDEs have been
explored, such as linearized FitzHugh-Nagumo equation \citep{Teresa2025Event} and the nonlinear Korteweg-de Vries equation \citep{Kang2021Event}.

Flexible structures, owing to their inherent advantages, have found widespread applications in engineering, where their vibrational behavior was often  modeled as beam equations \citep{Krstic2008,Cheng2022,Cheng2024,Cheng2025,Cheng2026}. As an efficient control method, event-triggered control has been adopted to address the stabilization problem of flexible beam systems \citep{Wu2024robust, Fan2025Event}. These investigates mainly apply the event-triggered control inside the system to suppress the system's vibrations, while few studies involve designing event-triggered control at the boundary of the beam. Motivated by this, we design two event-triggered control strategies implemented on the boundary to suppress the vibration of the following Rayleigh beam system
\begin{equation}\label{2021}
w_{tt}(x,t)+w_{xxxx}(x,t)-w_{xxtt}(x,t)=0,
\end{equation}
for all $\!x\!\in\![0,1]$ and $\!t\!>\!0$, where $w(x,t)$ represents the transversal deflection of beam at the position $x$ and time $t$ and $(\cdot)_{tt}=\frac{\partial^{2}(\cdot)}{\partial t^{2}}$,~$(\cdot)_{xt}=\frac{\partial(\cdot)}{\partial x\partial t}$. Assume that the left end of the beam is fixed, while
boundary control is applied at the right end. In this paper, the system (\ref{2021}) is equipped with the following boundary conditions:
\begin{equation}\label{uk}
\left\{\begin{array}{l}
     w_{xx}(1,t)=\mathbb{U}_{1}(t),\\
     w_{xtt}(1,t)-w_{xxx}(1,t)=\mathbb{U}_{2}(t),\\
     w(0,t)=w_{x}(0,t)=0,
\end{array}\right.
\end{equation}
for all $t>0$ and $\mathbb{U}_{i}(t)(i=1,2)$ denote the control input.

Notably, various types of control strategies have been proposed to address the boundary stabilization problem of Rayleigh beams, for instance, linear boundary controls \citep{wang2006stability}, dynamic boundary controls \citep{pavlovic2015dynamic,olotu2023free} and nonlinear boundary controls \citep{Xu2025variable}.
Compared to other control approaches, event-triggered controls adopted in this paper offers the advantage of reducing information update times.

The rest of the presented paper is organized as follows. Section 2 focuses on the decay estimates of energy for the resulting closed-loop system. The proof of the main result is presented in Section 3. Section 4 provides several simulation results to verify the theoretical findings. Finally, a concise conclusion is given in Section 5.
\vspace{10pt}
\section{Main results}\label{Se2}
\vspace{10pt}
The symbol $L^2(0,1)$ denotes the Hilbert space of all square-integrable functions defined on $[0,1]$, equipped with the norm
\begin{equation}
\|w\|=\bigg(\int_0^1w^2(x)dx\bigg)^{\frac{1}{2}}.
\end{equation}

In this paper, the event-triggered feedback laws in (\ref{uk}) are given by
\begin{eqnarray}\label{2022}
\mathbb{U}_{1}(t)=-K_1w_{xt}(1,t_k),~\mathbb{U}_{2}(t)=-K_2w_t(1,t_k),
\end{eqnarray}
 for any $t\in[t_k,t_{k+1}]$, where $K_1$ and $K_2$ are positive parameters, $w_{t}(1,t_k)$ and $w_{xt}(1,t_k)$ denote the transverse velocity and angular velocity of the right boundary of beam at time  $t_k$ ($k \in \mathbb{N}$), respectively.
Note that $t_k$ ($k \in \mathbb{N}$) are the triggering instants satisfying $0 = t_0 < t_1 < \cdots < t_k < t_{k+1}\cdots$.
The execution times are triggered by events generated according to a rule based on the system's energy, called the event-triggering mechanism, defined as:
\begin{equation}\label{tk}
\begin{aligned}
t_{k+1}:= \inf \Big\{ t \geq t_k :\max\{e^2_k,\hat{e}^2_k\}\!\geq\!\beta E(t)\!+\!\beta_{0}E(0)e^{-2\theta t} \Big\},
\end{aligned}
\end{equation}
where $e_k:=w_t(1,t)-w_t(1,t_k)$, $\hat{e}_k:=w_{xt}(1,t)-w_{xt}(1,t_k)$, for all $t\geq t_k(k\in\mathbb{N})$ and $\beta$, $\beta_{0}$ and $\theta$ are designed parameters to be determined later. The energy function $E(t)$ in (\ref{tk}) is defined as
\begin{eqnarray}\label{Et}
E(t):=\frac{1}{2}\int_0^1w^2_tdx+\frac{1}{2}\int_0^1w^2_{xx}dx+\frac{1}{2}\int_0^1w^2_{xt}dx.
\end{eqnarray}

\begin{rem}
In (\ref{Et}),
the first term  represents the kinetic
energy of the system, the second term
 corresponds to the elastic potential
energy due to bending, and the third term
 accounts for the kinetic energy
associated with the bending rate or higher-order inertial effects. 
\end{rem}
\begin{rem}
Event-triggered mechanisms involving the internal state error and the system energy function have been employed in \cite{Baudouin2019,Kang2021Event,Koudohode2022Event,Fan2025Event}. In addition, \cite{Espitia2016Event} and \cite{Gao2022Event} 
respectively designed boundary triggering control laws for  linear hyperbolic systems and flexible mechanical
systems to address system stability, but their triggering mechanisms rely on only one boundary control input.
 It should be noted that, the event-triggered mechanism $(\ref{tk})$,  based on two triggering control laws, is proposed for implementation at the boundary of the Rayleigh beam.
\end{rem}

\begin{rem}
 The exponential term $\beta_0E(0)e^{-2\theta t}$ in the triggering mechanism (\ref{tk}) ensures that the inter-execution time admits a positive lower bound over a finite time interval without the need to define an additional dwell time.
\end{rem}

For brevity, we employ the following abbreviated notations: $w := w(x,t)$, $w_{x} := w_{x}(x,t)$, $w_{t} :=w_t(x,t)$.
The closed-loop system resulting from (\ref{2021}) under the influence of feedback (\ref{2022}) is given by
\begin{numcases}{}
 w_{tt}+w_{xxxx}-w_{xxtt}=0, \label{eqsystem1}\\
     w_{xx}(1,t)=-K_1w_{xt}(1,t_k), \label{eqsystem2}\\
     w_{xtt}(1,t)-w_{xxx}(1,t)=-K_2w_t(1,t_k), \label{eqsystem3}\\
     w(0)=w_{x}(0)=0, \label{eqsystem4}\\
      w(x,t_k)=\iota_{k}(x), \; w_{t}(x,t_k)=\varsigma_{k}(x),\label{eqsystem5}
      \end{numcases}
where $\iota_k(x)$ and $\varsigma_k(x)$ represent the initial position and velocity at time $t_k$, respectively, for all $x \in [0,1]$ and $t \in [t_k,t_{k+1}]$ with $k \in \mathbb{N}$.


The main objective of this paper is to design an event-triggered mechanism (\ref{tk}), and prove that the closed-loop system (\ref{eqsystem1})-(\ref{eqsystem5}) is exponentially stable under this mechanism.
Naturally, the well-posedness of the system is equally crucial for the validity of the overall results. Due to page constraints, only the key ideas of the proof for the well-posedness result are outlined herein, which are structured into three aspects as follows:

(I) The existence of solutions to the closed-loop system on the interval $[0,t_1]$ is established by combining the linear semigroup method with the Lipschitz perturbation approach \citep{Pazy1983}.

(II) By virtue of induction method \citep{Kang2021Event}, it can be deduced that the closed-loop system admits a unique solution on each time interval $[t_k,t_{k+1}]$ for all $k\in\mathbb{N}$.

(III) Avoidance of Zeno's behavior.
The term $\beta_0 E(0) e^{-2\theta t}$ in (\ref{tk}) is an exponentially decaying positive term, where $\beta_0 > 0,~E(0)>0$ (the initial energy of the system) are positive constants, and $\theta > 0$ is a design parameter related to the decay rate. According to the event-triggering condition in (\ref{tk}), the triggering error $e_k(t_k)=\hat{e}_k(t_k)=0$, therefore within any finite time $T$, one has
\begin{eqnarray}
\begin{aligned}
&\max\{|e_k(t_{k+1})-e_k(t_k)|^2,|\hat{e}_k(t_{k+1})-\hat{e}_k(t_k)|^2\} \\
&\geq\beta_0 E(0) e^{-2\theta t_{k+1}}\geq \beta_0 E(0) e^{-2\theta T}.
\end{aligned}
\end{eqnarray}
Then, leveraging the regularity of solution (uniformly  continuity) and the proof by contradiction \citep{Kang2021Event}, the occurrence of Zeno behavior is effectively excluded.

Now, we state our stability result.

\begin{thm}\label{th3.1}
Given a desired decay rate $\delta > 0$ and $K_1,~K_2>0$, there exist positive parameters $\mu>\max\{ \frac{K^2_2-2\varepsilon\lambda^2}{4K_2-2\lambda-4\lambda \varepsilon}, \frac{K^2_1-\lambda^2\varepsilon-2\lambda^2K^2_1}{4K_1-2\lambda-2\lambda\varepsilon-4\lambda K^2_1},\frac{\lambda}{2}\}$ with $0<\lambda< \lambda^*= \min\{\frac{2K_2}{2\varepsilon+1},\frac{K_1}{\sqrt{2K^2_1+\varepsilon}}, \frac{1}{3}\}$ and $\varepsilon=\frac{2K_2(1+\alpha)+2K_1(1-\alpha)}{3-2\alpha}$ $ (\frac{1}{2}<\alpha\!<\!1)$  such that the closed-loop system under the event-triggering mechanism (\ref{tk}) is exponentially stable, satisfying the inequality:
\begin{equation}\label{9998}
E(t) \leq \mathcal{G} e^{-\delta t} E(0), \quad \forall t > 0,
\end{equation}
where $\mathcal{G} = \frac{1 + 3\lambda}{1 - 3\lambda} \left( 1 + \frac{\hat{\delta}}{2\theta - \delta} \right)$ and $\delta=\frac{\lambda C-2\mu\beta}{1-3\lambda},$ with $ C=\min\{2\alpha\!-\!1,\frac{3-2\alpha}{4}\},0 \leq \beta < \frac{\lambda C}{2\mu} ,\hat{\delta}=\frac{2\mu\beta_0}{1-3\lambda}, \beta_0 > 0$ and $\theta > \frac{\delta}{2}$.
\end{thm}
\begin{rem}
From Theorem \ref{th3.1}, we can easily see that in order to achieve exponential stability, only the parameters need to be designed, and the range of the parameters is directly given, without satisfying the negative definite matrix condition \citep{Kang2021Event}.
Moreover, both the coefficient  $\mathcal{G}$ and the attenuation rate $\delta$ increase monotonically with respect to $\lambda\in (0,\lambda^*)$.
In order to achieve the desired decay rate $\delta$, the design order of these parameters is $\alpha,C, K_1,K_2, \varepsilon,\lambda,\mu, \theta,\beta,\beta_0,\hat{\delta}$   and $\mathcal{G}$.
\end{rem}
\vspace{10pt}
\section{Proof of main result}
\vspace{10pt}
In this section, the energy perturbation method is adopted to complete the proof for the exponential stability of the closed-loop system. For this, we need to define the
following Lyapunov functional
\begin{equation}\label{vt}
V(t):=E(t)+\lambda\varrho(t),
\end{equation}
where $ \lambda $ is a positive tuning parameter, the energy $E(t)$ is defined in (\ref{Et}) and
\begin{eqnarray}\label{pt}
\!\!\varrho(t)\!:=\!\!\int_0^1\!\!\Big(w_t(xw_x\!-\!\alpha w)\!+\!w_{xt}[(1\!-\!\alpha)w_x\!+\!xw_{xx}]\Big)dx
\end{eqnarray}
with $\frac{1}{2}<\alpha<1$.

Before proving the stability results, we need to first provide some auxiliary lemmas.
\begin{lemma}\label{3.1}
$V(t)$ defined in (\ref{vt}) and $E(t)$   given  in (\ref{Et}) satisfy the following inequality:
\begin{eqnarray}\label{3.53}
(1-3\lambda)E(t)\leq V(t)\leq(1+3\lambda)E(t),
\end{eqnarray}
for any $t >0$  with $0<\lambda<\frac{1}{3}.$
\end{lemma}
{\bf Proof.} By applying Young's inequality,  $|\varrho(t)|$ satisfies
\begin{eqnarray}\label{4et}
\begin{aligned}
|\varrho(t)|&\leq\frac{1}{2}\int_0^1w^2_tdx+\frac{1}{2}\int_0^1w^2_xdx+\frac{\alpha}{2}\int_0^1w^2dx\\
&+\frac{\alpha}{2}\int_0^1\!\!w^2_tdx\!+\!\frac{1-\alpha}{2}\int_0^1\!\!w^2_{xt}dx\!+\!\frac{1}{2}\int_0^1w^2_{xx}dx\\
&+\frac{1}{2}\int_0^1w^2_{xt}dx+\frac{1-\alpha}{2}\int_0^1w^2_xdx.
\end{aligned}
\end{eqnarray}
In view of Poincar\'{e} inequality $||w||^2\leq||w_{x}\|^2\leq||w_{xx}\|^2$ and $\frac{1}{2}<\alpha<1$, we have
\begin{eqnarray}\label{1.15}
\begin{aligned}
|\varrho(t)|&\leq\frac{1\!+\!\alpha}{2}\!\!\int_0^1\!\!w^2_tdx+\frac{3}{2}\!\!\int_0^1\!\!w^2_{xx}dx\!\!+\!\!\frac{2\!-\!\alpha}{2}\int_0^1\!\!w^2_{xt}dx\\
&\leq3E(t).
\end{aligned}
\end{eqnarray}
Then substituting (\ref{1.15}) into (\ref{vt}) yields (\ref{3.53}).
\begin{lemma}
The derivatives of the energy function
$E(t)$ defined in (\ref{Et}) satisfies
\begin{equation}\label{738}
\frac{\mathrm{d}E(t)}{\mathrm{d}t} = -K_1w_{xt}(1,t)w_{xt}(1,t_k)- K_2w_t(1,t)w_t(1,t_k),
\end{equation}
for any $t \in [t_k, t_{k+1}]$ and $k \in \mathbb{N}$.
\end{lemma}
{\bf Proof.}
First, we multiply both sides of (\ref{eqsystem1}) by $u$ and integrate over $x \in [0,1]$ by parts to obtain the following variational structure:
\begin{eqnarray}\label{19}
\begin{aligned}
&\int_0^1w_{tt}udx\!+\!\int_0^1w_{xx}u_{xx}dx+\int_0^1w_{xtt}u_xdx\\
&=u_x(1,t)w_{xx}(1,t)\!+\!u(1,t)(w_{xtt}(1,t)\!-\!w_{xxx}(1,t))\\
&-u_x(0,t)w_{xx}(0,t)\!-\!u(0,t)(w_{xtt}(0,t)\!-\!w_{xxx}(0,t)).
\end{aligned}
\end{eqnarray}
Applying the boundary conditions (\ref{eqsystem2})-(\ref{eqsystem4}) to (\ref{19}), we obtain
\begin{eqnarray}\label{bf}
\begin{aligned}
&\int_0^1w_{tt}udx\!+\!\int_0^1w_{xx}u_{xx}dx+\int_0^1w_{xtt}u_{x}dx\\
&=-K_1u_{x}(1,t)w_{xt}(1,t_k)- K_2u(1,t)w_t(1,t_k).
\end{aligned}
\end{eqnarray}

Replacing $u = w_t$ in (\ref{bf}) implies
\begin{eqnarray}
\begin{aligned}
&\int_0^1w_{tt}w_tdx\!+\!\int_0^1w_{xx}w_{xxt}dx+\int_0^1w_{xtt}w_{xt}dx\\
&=-K_1w_{xt}(1,t)w_{xt}(1,t_k)- K_2w_t(1,t)w_t(1,t_k),
\end{aligned}
\end{eqnarray}
which yields the derivative of $E(t)$.

\begin{lemma}\label{1.32}
The derivatives of
$\varrho(t)$ defined in (\ref{pt}) satisfies
\begin{equation}\label{777}
\begin{aligned}
\frac{\mathrm{d}\varrho(t)}{\mathrm{d}t}&=\frac{1}{2}w^2_t(1,t)-(\frac{1}{2}+\alpha)\int_0^1w^2_tdx
-(\alpha-\frac{1}{2})\int_0^1w^2_{xt}dx\\
&+\frac{1}{2}w^2_{xt}(1,t)-(\frac{3}{2}-\alpha)\int_0^1w^2_{xx}dx-\frac{1}{2}w^2_{xx}(1,t)\\
&-K_2w_x(1,t)w_t(1,t_k)+K_2\alpha w(1,t)w_t(1,t_k)\\
&-K_1(1-\alpha)w_x(1,t)w_{xt}(1,t_k)+K^2_1w^2_{xt}(1,t_k),
\end{aligned}
\end{equation}
for any $t \in [t_k, t_{k+1}]$ and $k \in \mathbb{N}$.
\end{lemma}
{\bf Proof.}
Replace $u$ with $xw_x-\alpha w (\frac{1}{2}<\alpha<1)$ in (\ref{bf}) to produce
\begin{eqnarray}\label{N1}
\begin{aligned}
&\underbrace{\int_0^1w_{tt}(xw_x-\alpha w)dx}_{\mathcal{N}_1}+\underbrace{\int_0^1w_{xtt}(xw_x-\alpha w)_xdx}_{\mathcal{N}_2}\\
&+\underbrace{\int_0^1w_{xx}(xw_x-\alpha w)_{xx}dx}_{\mathcal{N}_3}\\
=&\underbrace{-K_1w_{xt}(1,t_k)[(1-\alpha)w_x(1,t)
+w_{xx}(1,t)]}_{\mathcal{N}_4}\\
&\underbrace{-K_2w_t(1,t_k)[w_x(1,t)-\alpha w(1,t)]}_{\mathcal{N}_5}.
\end{aligned}
\end{eqnarray}
By performing integration by parts, we derive the
following expression:
\begin{eqnarray}\label{2.5}
&\int_0^1xw_{tt}w_xdx=\int_0^1[xw_tw_x]_tdx-\int_0^1xw_{xt}w_tdx\\ \nonumber
&=\int_0^1[xw_tw_x]_tdx-\frac{1}{2}w^2_t(1,t)+\frac{1}{2}\int_0^1w^2_tdx
\end{eqnarray}
and
\begin{eqnarray}\label{2.6}
&-\alpha\int_0^1w_{tt}wdx=-\alpha\int_0^1[ww_t]_tdx+\alpha\int_0^1w^2_tdx.
\end{eqnarray}
Adding equations (\ref{2.5}) and (\ref{2.6}), we can obtain
\begin{eqnarray}\label{2.4}
\begin{aligned}
&\mathcal{N}_1=\int_0^1[w_t(xw_x-\alpha w)]_tdx\\
&~~~~~~~+(\frac{1}{2}+\alpha)\int_0^1w_t^2dx-\frac{1}{2}w_t^2(1,t).
\end{aligned}
\end{eqnarray}
By computing the derivative of
$xw_x-\alpha w$ with respect to $x$ in $\mathcal{N}_2$, we can simplify
\begin{eqnarray}\label{n2}
\mathcal{N}_2=\int_0^1x w_{xx}w_{xtt}dx+\int_0^1(1-\alpha)w_xw_{xtt}dx.
\end{eqnarray}
Through integration by parts to (\ref{n2}), we obtain
\begin{eqnarray}\label{n2.1}
\begin{aligned}
&\int_0^1x w_{xx}w_{xtt}dx\\
&=\int_0^1[xw_{xt}w_{xx}]_tdx+\frac{1}{2}\int_0^1w^2_{xt}dx-\frac{1}{2}w^2_{xt}(1,t)
\end{aligned}
\end{eqnarray}
and
\begin{eqnarray}\label{n2.2}
\begin{aligned}
&\int_0^1(1-\alpha)w_xw_{xtt}dx\\
&=(1-\alpha)\int_0^1[w_xw_{xt}]_tdx\!-\!(1-\alpha)\int_0^1w^2_{xt}dx.
\end{aligned}
\end{eqnarray}
Summing (\ref{n2.1}) and (\ref{n2.2}) gives
\begin{equation}\label{2.61}
\begin{aligned}
\mathcal{N}_2&=\int_0^1[(1-\alpha)w_x w_{xt}+xw_{xt}w_{xx}]_tdx\\
&~~~+(\alpha-\frac{1}{2})\int_0^1w^2_{xt}dx-\frac{1}{2}w^2_{xt}(1,t).
\end{aligned}
\end{equation}

By taking the second-order derivative of $xw_x-\alpha w$  with respect to $x$ in $\mathcal{N}_3$, we are able to reorganize $\mathcal{N}_3$ as
\begin{eqnarray}
\mathcal{N}_3=\int_0^1(2-\alpha)w^2_{xx}dx+\int_0^1xw_{xx}w_{xxx}dx.
\end{eqnarray}
Subsequently, integration by parts results in
\begin{eqnarray}\label{2.42}
\begin{aligned}
\mathcal{N}_3=\int_0^1(\frac{3}{2}-\alpha)w^2_{xx}dx+\frac{1}{2}w^2_{xx}(1,t).
\end{aligned}
\end{eqnarray}
Applying boundary conditions (\ref{eqsystem2}), one can effortlessly derive
\begin{eqnarray}\label{2.43}
\mathcal{N}_4=-K_1(1-\alpha)w_x(1,t)w_{xt}(1,t_k)+K^2_1w^2_{xt}(1,t_k)
\end{eqnarray}
and
\begin{eqnarray}\label{2.44}
\mathcal{N}_5=-K_2w_x(1,t)w_t(1,t_k)+K_2\alpha w(1,t)w_t(1,t_k).
\end{eqnarray}

Ultimately, substituting the revised forms of $\mathcal{N}_1$ through $\mathcal{N}_5$ into (\ref{N1}) gives
\begin{eqnarray}
\begin{aligned}
&\!\int_0^1\!\![(1\!-\!\alpha)w_x w_{xt}\!+\!xw_{xt}w_{xx}]_tdx+\!\!\int_0^1\!\![w_t(xw_x-\alpha w)]_tdx\\
&=\frac{1}{2}w^2_t(1,t)-(\frac{1}{2}+\alpha)\int_0^1w^2_tdx
-(\alpha-\frac{1}{2})\int_0^1w^2_{xt}dx\\
&+\frac{1}{2}w^2_{xt}(1,t)-(\frac{3}{2}-\alpha)\int_0^1w^2_{xx}dx-\frac{1}{2}w^2_{xx}(1,t)\\
&-K_2w_x(1,t)w_t(1,t_k)+K_2\alpha w(1,t)w_t(1,t_k)\\
&-K_1(1-\alpha)w_x(1,t)w_{xt}(1,t_k)+K^2_1w^2_{xt}(1,t_k),
\end{aligned}
\end{eqnarray}
which leads to the desired result (\ref{777}).

From the above preliminary work, we will prove our stability result.

{\bf Proof of Theorem \ref{th3.1}.}
Since $e_{k}$ and $\hat{e}_{k}$ in (\ref{tk}) bounded on each time sub-interval $[t_k, t_{k+1}]$ and $k \in \mathbb{N}$, we have
\begin{equation}\label{166}
e_k^2+\hat{e}_k^2\leq 2\beta E(t) + 2\beta_0 E(0) e^{-2\theta t}.
\end{equation}

For $t \in [t_k, t_{k+1}]$ and $k \in \mathbb{N}$, differentiating $V(t)$ in (\ref{vt}) and using (\ref{738}), (\ref{777}) and (\ref{166}), we obtain
\begin{eqnarray}\label{lyp}
\begin{aligned}
\frac{\mathrm{d}V(t)}{\mathrm{d}t}&\leq\mu(2\beta E(t)+2\beta_0E(0)e^{-2\theta t}-e^2_{k}-\hat{e}^2_{k})\\
&~~~+\frac{\mathrm{d}E(t)}{\mathrm{d}t}+\lambda\frac{\mathrm{d}\varrho(t)}{\mathrm{d}t}\\
&\leq\underbrace{-\lambda(\frac{1}{2}
+\alpha)\int_0^1w^2_{t}dx}_{\mathcal{M}_1}
\underbrace{-\lambda(\alpha-\frac{1}{2})\int_0^1w^2_{xt}dx}_{\mathcal{M}_2}\\
&~~~\underbrace{-\lambda(\frac{3}{2}-\alpha)\int_0^1w^2_{xx}dx}_{\mathcal{M}_3}
+K^2_1w^2_{xt}(1,t_k)\\
&~~~\underbrace{-K_2w_x(1,t)w_t(1,t_k)+K_2\alpha w(1,t)w_t(1,t_k)}_{\mathcal{M}_4}\\
&~~~\underbrace{-K_1(1-\alpha)w_x(1,t)w_{xt}(1,t_k)}_{\mathcal{M}_5}-\mu w^2_t(1,t)\\
&~~~-\mu w^2_{xt}(1,t_k)+(2\mu\!-\!K_1)w_{xt}(1,t)w_{xt}(1,t_k)\\
&~~~+(2\mu\!-\!K_2)w_t(1,t)w_t(1,t_k)+\frac{\lambda}{2}w^2_t(1,t)\\
&~~~-\mu w^2_t(1,t_k)-\mu w^2_{xt}(1,t)+\frac{\lambda}{2}w^2_{xt}(1,t),
\end{aligned}
\end{eqnarray}
where $\mu>0$ is a parameter.

By virtue of Young's inequality ($ab\leq\frac{1}{2\varepsilon}a^2+\frac{\varepsilon}{2}b^2,\forall\varepsilon>0)$ for $\mathcal{M}_4$ and $\mathcal{M}_5$, it holds
\begin{eqnarray}\label{m41}
\begin{aligned}
\mathcal{M}_4&\leq\frac{K^2_2}{2\varepsilon}w^2_x(1,t)+\frac{K^2_2\alpha^2}{2\varepsilon}w^2(1,t)
+\varepsilon w^2_t(1,t_k)
\end{aligned}
\end{eqnarray}
and
\begin{eqnarray}\label{m51}
\mathcal{M}_5\leq\frac{K^2_1(1-\alpha)^2}{2\varepsilon}w^2_x(1,t)+\frac{\varepsilon}{2}w^2_{xt}(1,t_k).
\end{eqnarray}
In light of the Poincar\'{e} inequality, we have $w^2(1,t)\leq \|w_x\|^2\leq\|w_{xx}\|^2$ and $w^2_x(1,t)\leq\|w_{xx}\|^2$. Then it follows from (\ref{m41}) and (\ref{m51}) that
\begin{eqnarray}\label{m4}
\mathcal{M}_4\leq\frac{K^2_2(1+\alpha^2)}{2\varepsilon}\int_0^1w^2_{xx}dx+\varepsilon w^2_t(1,t_k)
\end{eqnarray}
and
\begin{eqnarray}\label{m5}
\mathcal{M}_5\leq\frac{K^2_1(1-\alpha)^2}{2\varepsilon}\int_0^1w^2_{xx}dx+\frac{\varepsilon}{2}w^2_{xt}(1,t_k).
\end{eqnarray}
Calculate the sum of $\mathcal{M}_1$ to $\mathcal{M}_5$ and invoke (\ref{m4}) and (\ref{m5}) to obtain
\begin{eqnarray}\label{m1m5}
\begin{aligned}
&\mathcal{M}_1+\cdots+\mathcal{M}_5\\
&\leq-\lambda\bigg(\frac{1}{2}+\alpha\bigg)\int_0^1w^2_tdx-
\lambda(\alpha-\frac{1}{2})\int_0^1w^2_{xt}dx\\
&~~~-\lambda\bigg(\frac{3}{2}\!-\!\alpha\!-\!\frac{K^2_2(1\!+\!\alpha^2)+K^2_1(1\!-\!\alpha)^2}{2\varepsilon}\bigg)
\int_0^1w^2_{xx}dx\\
&\leq-\lambda CE(t),
\end{aligned}
\end{eqnarray}
where $\lambda>0$ and $C=\min\{2\alpha-1,\frac{3-2\alpha}{4}\}$. Since $\varepsilon>0$ is an arbitrary Young's parameter, we let $\varepsilon=\frac{2K^2_2(1+\alpha^2)+2K^2_1(1\!-\!\alpha)^2}{3-2\alpha}$  guarantee  ``$\frac{3}{2}-\alpha-\frac{K^2_2(1+\alpha^2)+K^2_1(1-\alpha)^2}{2\varepsilon}$ $=\frac{3-2\alpha}{4}$ '', which implies that (\ref{m1m5}) holds.

Substituting (\ref{m1m5}) into (\ref{lyp}) and combining like terms, we can show that
\begin{eqnarray}\label{6794}
\begin{aligned}
\frac{\mathrm{d}V(t)}{\mathrm{d}t}&\!\leq\!(\frac{\lambda}{2}-\mu)w^2_t(1,t)+(2\mu-K_2)w_t(1,t)w_t(1,t_k)\\
&~~~+(\lambda\varepsilon-\mu)w^2_t(1,t_k)+(2\mu\beta-\lambda C)E(t)\\
&~~~+(2\mu-K_1)w_{xt}(1,t)w_{xt}(1,t_k)\\
&~~~+(\frac{\varepsilon\lambda}{2}+\lambda K^2_1-\mu)w^2_{xt}(1,t_k)\\
&~~~+(\frac{\lambda}{2}-\mu)w^2_{xt}(1,t)+2\mu\beta_0E(0)e^{-2\theta t}\\
&\leq
\Upsilon^\top\mathcal{D}_1\Upsilon+\Gamma^\top\mathcal{D}_2\Gamma
+(2\mu\beta-\lambda C)E(t)\\
&~~~+2\mu\beta_0E(0)e^{-2\theta t},
\end{aligned}
\end{eqnarray}

where $\Upsilon\!\!=\!\!(w_t(1,t),w_t(1,t_k))^\top, \Gamma\!\!=\!\!(w_{xt}(1,t),w_{xt}(1,t_k))^\top$,
\begin{eqnarray}
\mathcal{D}_1\!=\!
\begin{pmatrix}
\frac{\lambda}{2} - \mu & \mu - \frac{K_2}{2} \\
\mu - \frac{K_2}{2} & \lambda\varepsilon - \mu
\end{pmatrix}
\end{eqnarray}
and
\begin{eqnarray}
\mathcal{D}_2\!=\!
\begin{pmatrix}
\frac{\lambda}{2} - \mu & \mu - \frac{K_1}{2} \\
\mu - \frac{K_1}{2} & \frac{\lambda\varepsilon}{2}+\lambda K^2_1 - \mu
\end{pmatrix}.
\end{eqnarray}
Given $\mu\!>\!\max\{ \frac{K^2_2-2\varepsilon\lambda^2}{4K_2-2\lambda-4\lambda \varepsilon}, \frac{K^2_1-\lambda^2\varepsilon-2\lambda^2K^2_1}{4K_1-2\lambda-2\lambda\varepsilon-4\lambda K^2_1},\frac{\lambda}{2}\}$ with $0\!\!<\!\!\lambda\!\!<\!\! \min\{\frac{2K_2}{2\varepsilon+1},\frac{K_1}{\sqrt{2K^2_1+\varepsilon}}, \frac{1}{3}\}$ such that $\mathcal{D}_1$ and $\mathcal{D}_2$ are negative definite matrices. Then the inequality (\ref{6794}) becomes
\begin{equation}\label{35}
\begin{aligned}
\frac{\mathrm{d}V(t)}{\mathrm{d}t}\leq(2\mu\beta-\lambda C)E(t)+2\mu\beta_0E(0)e^{-2\theta t},
\end{aligned}
\end{equation}
with $0<\lambda<\frac{1}{3}$.

Substituting (\ref{3.53}) into  (\ref{35}) leads to
\begin{eqnarray}\label{48}
\frac{\mathrm{d}V(t)}{\mathrm{d}t}\leq\frac{(2\mu\beta-\lambda C)}{1-3\lambda}V(t)+\frac{2\mu\beta_0}{1-3\lambda}V(0)e^{-2\theta t},
\end{eqnarray}
for all $t>0$.

Let  $\delta=\frac{\lambda C-2\mu\beta}{1-3\lambda}<2\theta, \hat{\delta}=\frac{2\mu\beta_0}{1-3\lambda}$ with $0\leq \beta < \frac{\lambda C}{2\mu}$, then for any $t>0$, applying Gronwall inequality to (\ref{48}), one has
\begin{eqnarray}\label{88}
\begin{aligned}
V(t) &\leq e^{-\delta t} V(0) + \hat{\delta} e^{-\delta t} V(0) \int_0^t e^{(-2\theta + \delta)s} \, ds \\
&\leq e^{-\delta t} V(0) + \frac{\hat{\delta}}{2\theta - \delta} V(0) (e^{-\delta t} - e^{-2\theta t}) \\
&\leq \left( 1 + \frac{\hat{\delta}}{2\theta - \delta} \right) e^{-\delta t} V(0) - \frac{\hat{\delta}}{2\theta - \delta} e^{-2\theta t} V(0).
\end{aligned}
\end{eqnarray}
From (\ref{3.53}), it follows that
\begin{equation}\label{estimate}
E(t) \leq \mathcal{G} e^{-\delta t} E(0), \quad \forall t > 0,
\end{equation}
where $\mathcal{G} = \frac{1 + 3\lambda}{1 - 3\lambda} \left( 1 + \frac{\hat{\delta}}{2\theta - \delta} \right)$, completing the proof of Theorem \ref{th3.1}. \hfill$\Box$

\vspace{10pt}
\section{Numerical Simulations}
\vspace{10pt}
In this section, a simulation example of the closed-loop system (\ref {eqsystem1})-(\ref {eqsystem5}) is presented to verify the effectiveness of the proposed control law (\ref {2022}), and the finite difference method is adopted for the simulation.

Given parameters $K_1=0.2, K_2=0.1$, the total time $T=2$ and the initial state $\iota_0(x)=x-\mathrm{sin}x,\varsigma_0(x)=1-\mathrm{cos}x$, we choose $\beta=0.01, \beta_0=0.005,\theta=0.2$ for the event-triggering mechanism (\ref{tk}) in the simulation.


From Fig. 1, it is easy to find that the closed-loop system (\ref{eqsystem1})-(\ref{eqsystem5}) under event-triggering mechanism (\ref{tk}) can be rapidly stabilized, which indicates that the transverse
vibration of Rayleigh beam has been effectively suppressed.
Moreover,
The energy function $E(t)$ of the closed-loop system (\ref{eqsystem1})-(\ref{eqsystem5}) under event-triggering mechanism (\ref{tk}) decays to zero, as shown in Fig. 2. Additionally, the release time and release interval are clearly described in Fig. 3 when implementing event triggered control.
 Fig. 4 and Fig. 5  depict the
evolution of the magnitude of the controller $\mathbb{U}_1(t)$ and $\mathbb{U}_2(t)$  respectively
in the continuous-in-time and the event-triggering
frameworks. From the perspective of conserving communication resources, the energy-saving performance of event-triggered control is superior to that of continuous-time control.
\vspace{10pt}
\section{Conclusion}
\vspace{10pt}
This paper formulates two event-triggered control laws to stabilize Rayleigh beam systems. A novel event-triggered control mechanism is proposed and the exponential stability of the closed-loop system is guaranteed via the integral multiplier technique and the energy perturbation method, thereby achieving the expected decay rate. In the future, we will consider designing event triggered control for nonlinear beam systems and coupled systems at the boundary to solve the  stabilization problem.


 \begin{figure}
 \centering
 \includegraphics[width=1\columnwidth]{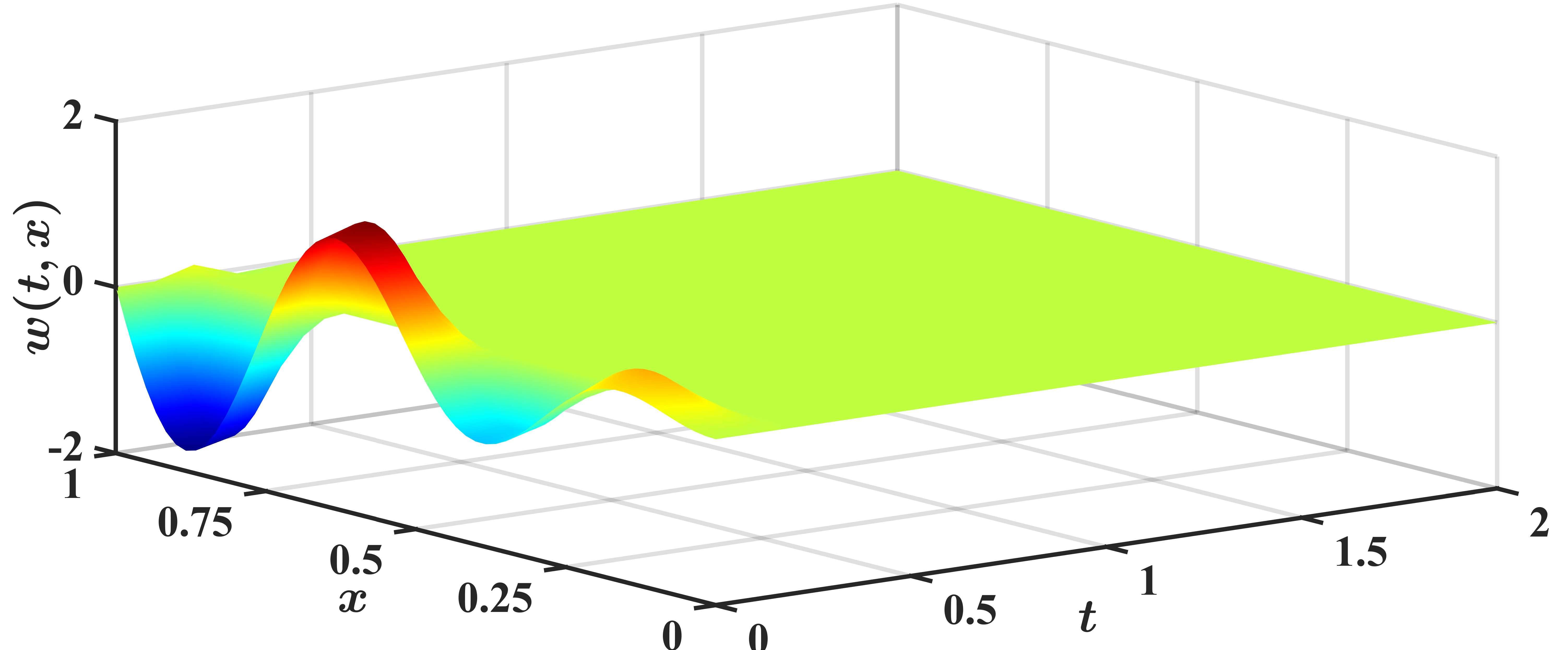}
\caption{Transverse displacements of the closed-loop system
with the event-triggered control.}
\label{figure1}
\end{figure}

\begin{figure}
 \centering
 \includegraphics[width=1\columnwidth]{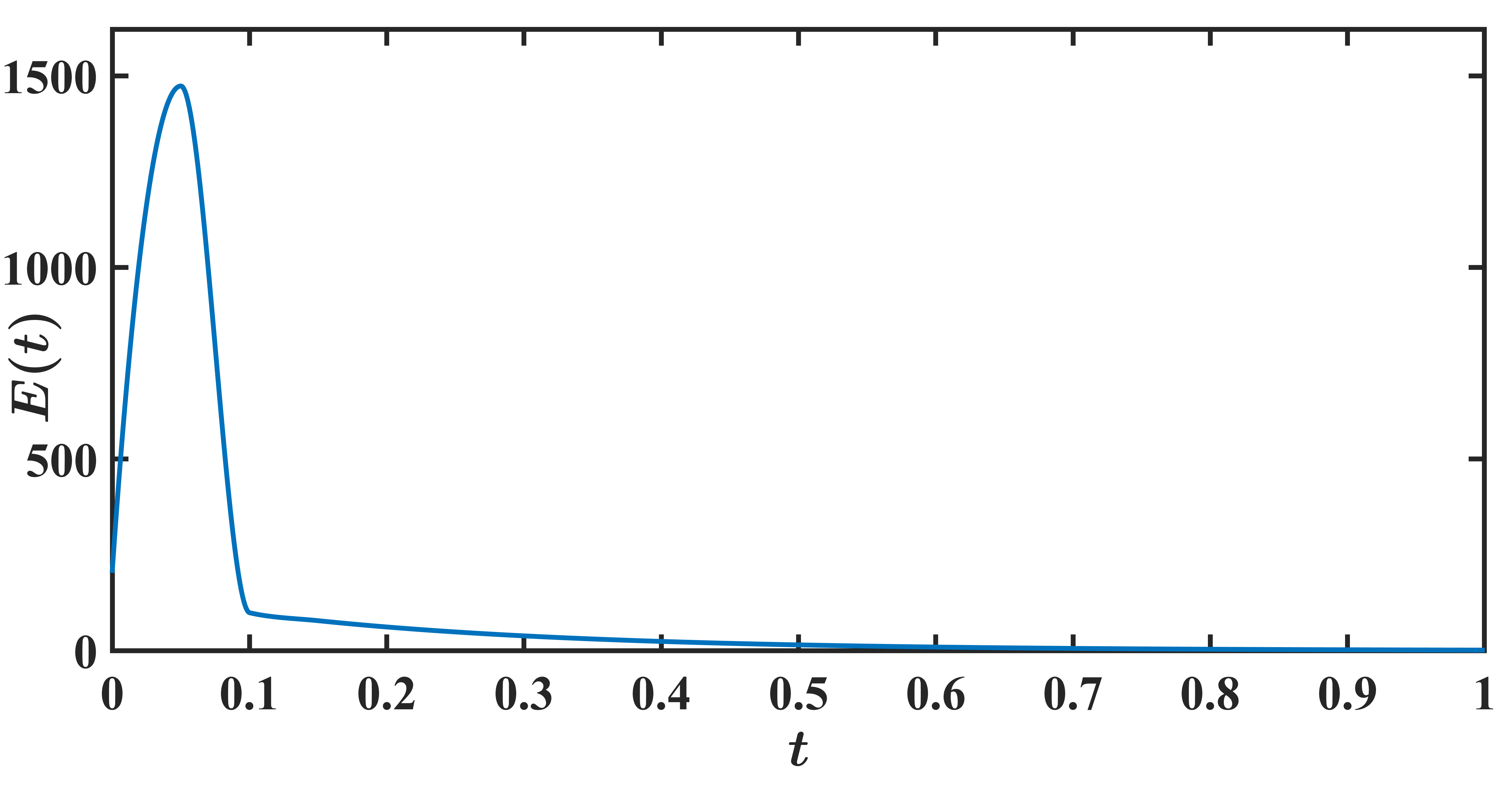}
\caption{Time evolution of the energy $E(t)$ of the closed-loop system
under the event-triggered control.}
\label{figure2}
\end{figure}

\begin{figure}
 \centering
\includegraphics[width=1\columnwidth]{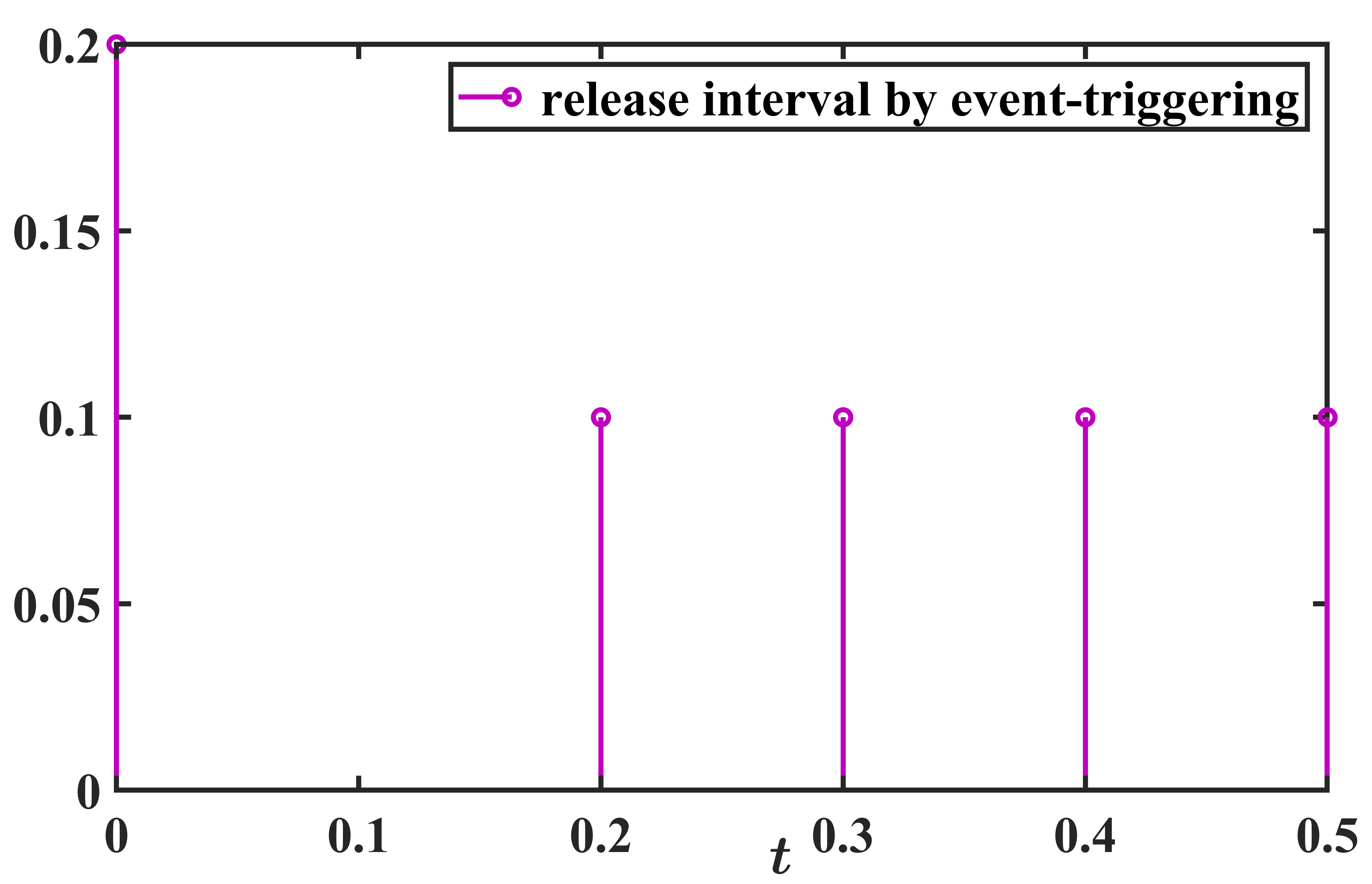}
\caption{Release instants and release interval by event-triggering.}
\label{figure3}
\end{figure}

\begin{figure}
 \centering
\includegraphics[width=1\columnwidth]{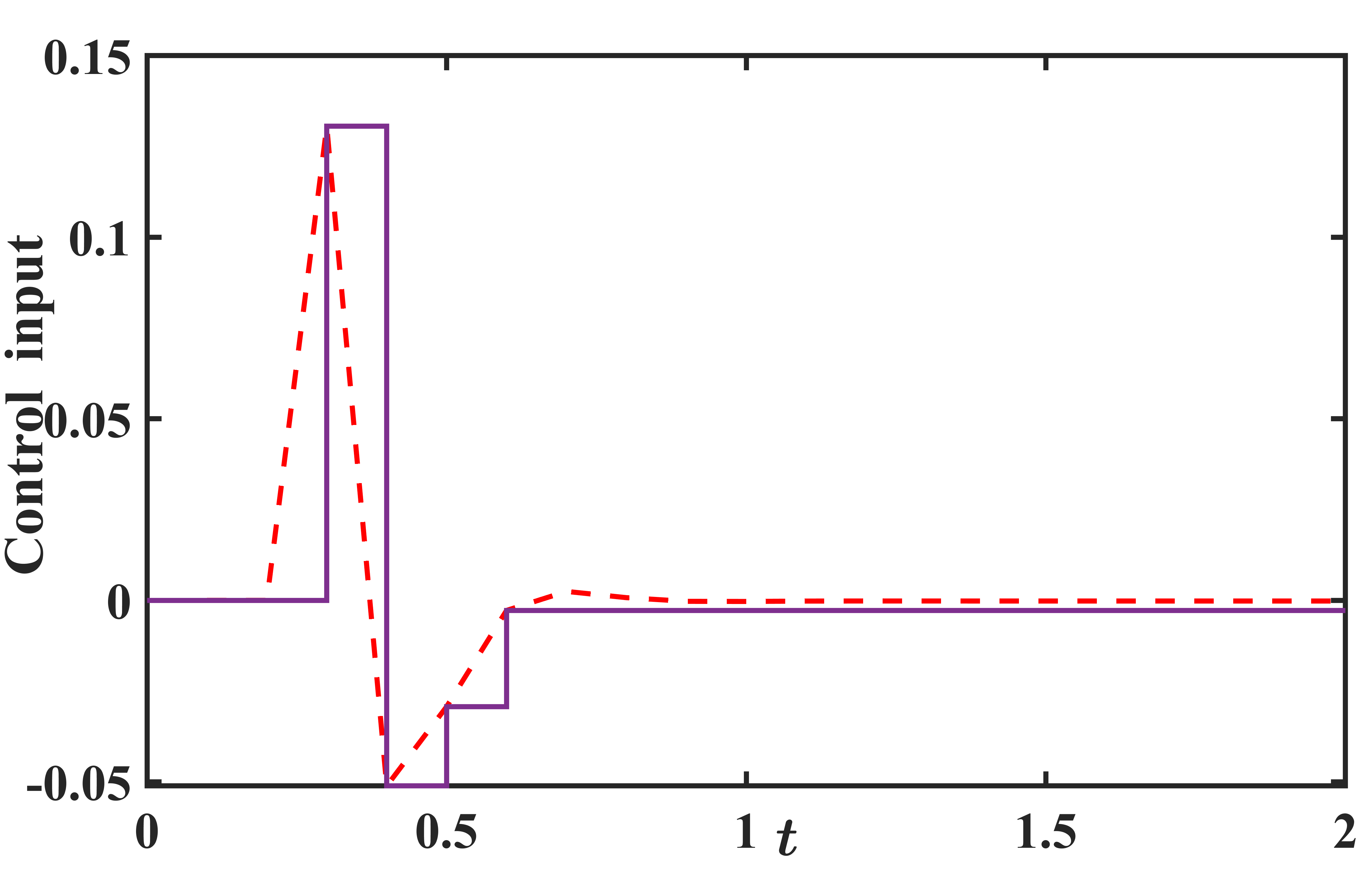}
\caption{Evolution of  $\mathbb{U}_1(t)$ for the continuous-time controller in red dashed line and the event-triggered controller in purple line.}
\label{figure4}
\end{figure}
 \begin{figure}
 \centering
\includegraphics[width=1\columnwidth]{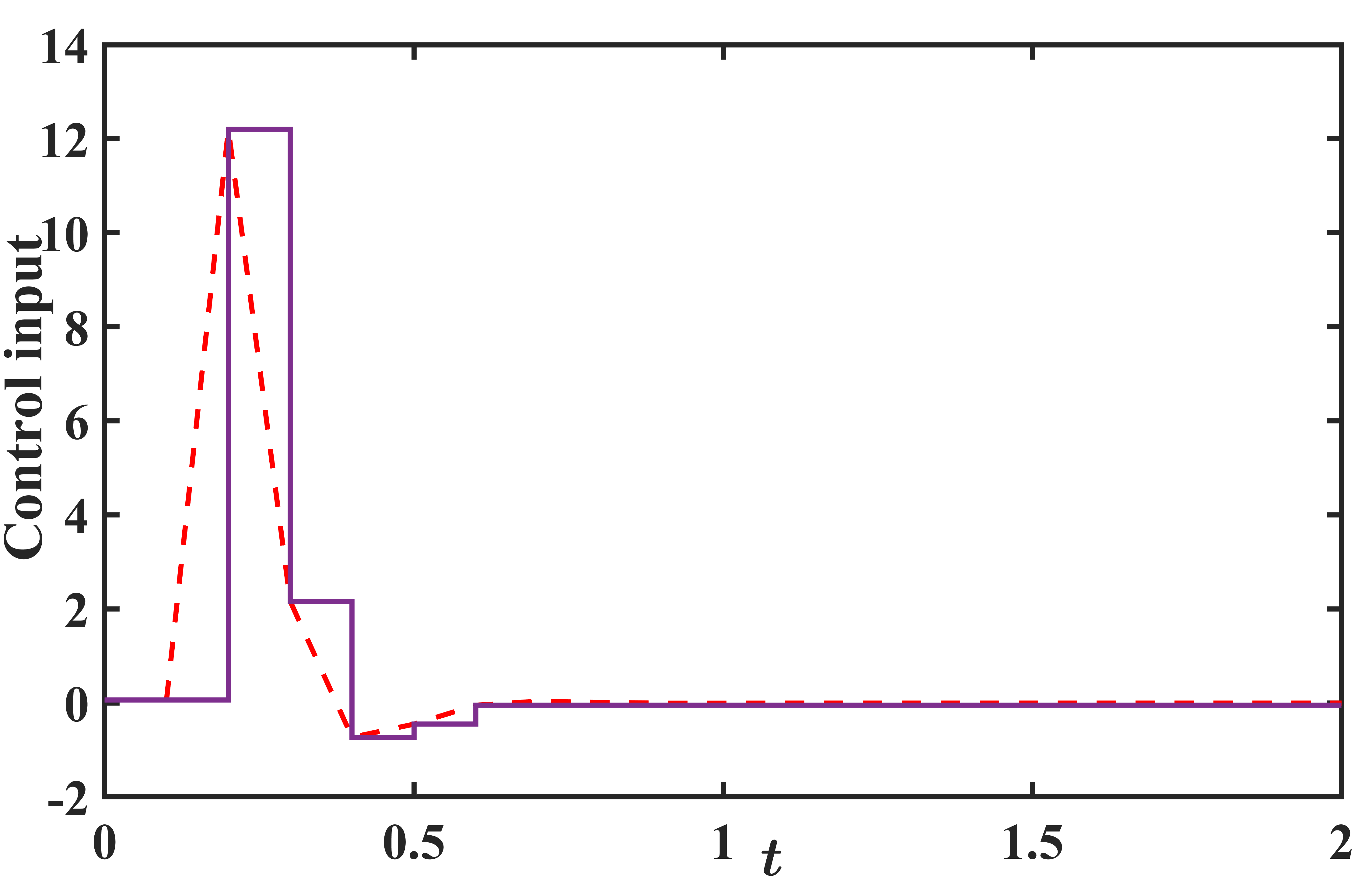}
\caption{Evolution of  $\mathbb{U}_2(t)$ for the continuous-time controller in red dashed line and the event-triggered controller in purple line.}
\label{figure5}
\end{figure}

  \vspace{10pt}                                                   
  \bibliography{ifacconf}   
\end{document}